\documentclass[11pt,article]{article}

\usepackage[T1]{fontenc}
\usepackage[english]{babel}
 
\usepackage{amssymb,amsfonts,amsmath,mathtext,amsthm, geometry,cite,euler}
\usepackage[hidelinks]{hyperref}
\hypersetup{colorlinks=true,urlcolor=blue,linkcolor=blue,citecolor=red}

\newcommand{\n}[1]{\left\|#1 \right\|} 
\renewcommand{\a}{\alpha}
\renewcommand{\b}{\beta}
\newcommand{\la}{\lambda}
\newcommand{\x}{\bar x}

\newcommand{\R}{\mathbb R}
\newcommand{\N}{\mathbb N}

\newcommand{\Z}{\mathbb Z}

\renewcommand {\r}{\rightarrow}
\newcommand{\rr}{\rightharpoonup} 

\newtheorem{theorem}{Theorem} 

\theoremstyle{definition}
 
\newtheorem{lemma}{Lemma} 
\newtheorem{algorithm}{Algorithm}
\newtheorem{remark}{Remark}

\begin{document}

\title{A hybrid method without extrapolation step for solving variational inequality problems}

\author{Yu. V. Malitsky\thanks{Department of Cybernetics, Taras Shevchenko National University of
Kyiv,  64/13,  Volodymyrska Str.,  Kyiv, 01601,  Ukraine. E-mail: \texttt{y.malitsky@gmail.com}.}\and
V. V. Semenov \thanks{ Department of Cybernetics, Taras Shevchenko National University of Kyiv 
64/13,  Volodymyrska Str.,  Kyiv, 01601, Ukraine. E-mail: \texttt{semenov.volodya@gmail.com}. }}

\maketitle

\begin{abstract} \noindent
In this paper, we introduce a new method for solving variational inequality problems with monotone and Lipschitz-continuous mapping in Hilbert space. The iterative process is based on two well-known projection method and the hybrid (or outer approximation) method. However we do not use an extrapolation step in the projection method. The absence of one projection in our method is explained by slightly different choice of sets in hybrid method. We prove a strong convergence of the sequences generated by our method.
\end{abstract}

{\small
 
}
 \section{Introduction}
\label{intro}
We consider the classical variational inequality problem which is to find a point $x^*\in C$ such that
\begin{equation}\label{vip}
(Ax^*, x-x^*)\geq 0 \quad \forall x\in C,
\end{equation}
where $C$ is a closed convex set in Hilbert space $H$, $(\cdot, \cdot)$ denotes the inner product in $H$, and $A\colon H\r H$ is a some mapping.
We assume that the following conditions hold\\
(C1)\quad The solution set of \eqref{vip}, denoted by $S$, is nonempty.\\
(C2)\quad The mapping $A$ is monotone, i.e.,
$$(Ax-Ay,x-y)\geq 0\quad \forall x,y \in C.$$
(C3)\quad  The mapping $A$ is Lipschitz-continuous on $C$ with constant $L>0$, i.e., there exists $L>0$, such that
$$\n{Ax-Ay}\leq L \n{x-y} \quad \forall x,y\in C.$$

Variational inequality theory is an important tool in studying a wide class of obstacle,
unilateral, and equilibrium problems arising in several branches of pure and applied sciences in
a unified and general framework \cite{aubinekeland,baiocchi:84,lions:81,kinderlehrer,konnov2001,nagurney99}. This field is dynamic and is experiencing an explosive growth in both theory and applications. Several numerical methods have been developed for solving variational inequalities
and related optimization problems, see books \cite{baucomb,pang:03,konnov2001} and the references therein. 
In order to construct an algorithm which provides  strong convergence to a solution of \eqref{vip} we propose the following method \\
\textbf{Algorithm 1}
\begin{equation}\label{our}
\begin{cases}
x_0,z_0\in C,\\
z_{n+1} = P_C(x_n-\la Az_n),\\
x_{n+1} = P_{C_n\cap Q_n} x_0,
\end{cases}
\end{equation}
where $P_M$ denotes the metric projection on the set $M$, $\la \in (0,1/2L)$, the sets $C_n$ and $Q_n$ are some halfspaces which we will define  in the Section \ref{algorithm}.

\section{Relation to the previous work}
There are two general approaches to study variational inequality problem under the conditions above:  regularization methods and projection methods. As far as we are concerned, we prioritize the latter; good survey of regularization methods can be found in \cite{bakush,konnov2001}.
 
The oldest algorithm which provides convergence of the generated sequence under the above assumptions is the extragradient method proposed in 1976 by Korpelevich~\cite{korpel:76} and Antipin~\cite{antipin}. A lot of efficient modifications exist at this moment \cite{apostol12,ceng:10,reich:2011,iusem:00,Iusem:97,khobotov:89,lyashko11,nadezhkina:06,solodov:1999,tseng00,voitova11}. The natural question that arises in the case of infinite dimensional Hilbert space is how to construct an algorithm which provides strong convergence.  To answer this question Nadezhkina and Takahashi~\cite{nadezhkina:06} introduced the following method \\
\textbf{Algorithm 2}
\begin{equation}\label{nadezhda}
\begin{cases}
x_0\in C,\\
y_n = P_C(x_n-\la Ax_n),\\
z_n = P_C(x_n-\la Ay_n),\\
C_n=\{w\in C \colon \n{z_n-w}\leq \n{x_n-w}\},\\
Q_n = \{w\in C \colon (x_n-w, x_0- x_n)\geq 0 \},\\
x_{n+1} = P_{C_n\cap Q_n} x_0,
\end{cases}
\end{equation}
where $\la\in (0,1/L)$.
Under the above assumptions (C1)--(C3) they proved that the sequence $(x_n)$, generated by \eqref{nadezhda}, converges strongly to $P_S x_0$. Their method based on the extragradient method  and on the hybrid method, proposed by \cite{nakajo03,solodov00}. The computational complexity of \eqref{nadezhda} on every step is three computation of metric projection and two values of $A$.

Inspired by this scheme, Censor, Gibali and Reich \cite{censor_reich:2011} presented the following algorithm \\
\textbf{Algorithm 3}
\begin{equation}\label{censor_strong}    
\begin{cases}
x_0\in H,\\
y_n = P_C(x_n-\la Ax_n),\\
T_n = \{w\in H \colon (x_n-\la Ax_n-y_n,w-y_n)\leq 0\},\\
z_n = \a_n x_n + (1-\a_n) P_{T_n}(x_n-\la Ay_n),\\
C_n=\{w\in H \colon \n{z_n-w}\leq \n{x_n-w}\},\\
Q_n = \{w\in H \colon (x_n-w, x_0- x_n)\geq 0 \},\\
x_{n+1} = P_{C_n\cap Q_n} x_0,
\end{cases}
\end{equation}
where $\la \in (0,1/L)$ and $(\a_n)\subset [0,\a]$ for some $\a \in [0,1)$.
In contrast to \eqref{nadezhda} the sets $C_n$ and $Q_n$ are halfspaces and hence it is much more simpler to calculate  $P_{C_n\cap Q_n}x_0$ than projection onto the general convex set $C$. Therefore in the next schemes  we will not take into consideration this projection. Also on the second step it is calculated only  a projection onto halfspace $T_n$ but  not onto set $C$ like in \eqref{nadezhda}. However on every step of \eqref{censor_strong} we need to calculate two values of $A$ as well as in \eqref{nadezhda}. 

Using hybrid method it is not difficult to modify Tseng algorithm \cite{tseng00} \\ \textbf{Algorithm 4}
\begin{equation}\label{tseng_strong}
\begin{cases}
x_0\in H,\\
y_n = P_C(x_n-\la Ax_n),\\
z_n =y_n + \la (Ax_n-Ay_n),\\
C_n=\{w\in H \colon \n{z_n-w}\leq \n{x_n-w}\},\\
Q_n = \{w\in H \colon (x_n-w, x_0- x_n)\geq 0 \},\\
x_{n+1} = P_{C_n\cap Q_n} x_0,
\end{cases}
\end{equation}
where $\la \in (0,1/L)$. In computational sense it is similar to \eqref{censor_strong}.

In this work we show that with some other choice of sets $C_n$ it is possible to throw out in \eqref{nadezhda} or in \eqref{censor_strong} the step of extrapolation which consist in $y_n = P_C(x_n - \la Ax_n)$.  We emphasize that in order to prove a convergence of Algorithm 1 we use an idea due to Popov~\cite{popov:80} (the refinement of this idea see also in \cite{mal-sem}). It is easy to see that our method~\eqref{our}  on every iteration needs only one computation of projection (as in \eqref{censor_strong} or \eqref{tseng_strong}) and only one value of $A$. For example,  very often variational inequality problems which arise from optimal control,  provide a very complicated operator such that only computation of the  latter is a very sophisticate problem. (For more details see \cite{lions71}).

\section{Preliminaries}
\label{pre}
In order to prove our main result, we need the following statements (see books \cite{aubinekeland,baiocchi:84,baucomb,kinderlehrer}). At first, the following well-known properties of the projection mapping will be used throughout this paper.
\begin{lemma}\label{proj}
Let $M$ be nonempty closed convex set in $H$, $x \in H$. Then
 \begin{itemize}
  \item[i)] $(P_Mx - x, y-P_Mx)\geq 0\quad \forall y\in M$;
  \item[ii)] $\n{P_Mx-y}^2\leq \n{x-y}^2 - \n{x-P_Mx}^2 \quad \forall y\in M$.
 \end{itemize}
\end{lemma}

Next two lemmas are also well-known.
\begin{lemma}[Minty]\label{minty}
Assume that $A:C\r H$ is a continuous and monotone mapping. Then $x^*$ is a solution of \eqref{vip} iff $x^*$ is a solution of the following problem
$$\text{find}\, x\in C, \quad \text{such that}\quad (Ay,y-x)\geq 0\quad \forall y\in C.$$
\end{lemma}
\begin{remark}
The solution set $S$ of variational inequality \eqref{vip}  is closed and convex.
\end{remark}

We write $x_n \rightharpoonup x$ to indicate that the sequence $(x_n)$ converges weakly to $x$. $x_n \to x$ implies that $(x_n)$ converges strongly to $x$. 
\begin{lemma}[Kadec--Klee property of a Hilbert space]\label{kadec}
Let $(x_n)$ be a sequence in $H$. Then from  $\n{x_n}\r \n{x}$ and $x_n\rr x$ follows that $x_n \r x$.
\end{lemma}

At last, we prove the following result.
\begin{lemma} \label{lema_tech}
 Let $(a_n)$, $(b_n)$, $(c_n)$ be nonnegative real sequences, $\a, \b\in \R$  and for all $n\in \N$ the following inequality holds
  \begin{equation}\label{tech_ineq}
a_n \leq b_n - \a c_{n+1} +\b c_n.
  \end{equation}
If $\sum_{n=1}^\infty b_n< +\infty$ and $\a> \b \geq 0$ then $\lim_{n\r \infty} a_n = 0$.
\end{lemma}
\begin{proof}
   Using inequality \eqref{tech_ineq} for $n=1, n=2, \dots, n=N$ we obtain
  \begin{gather*}
    a_1 \leq b_1 - \a c_{2} + \b c_1,\\
    a_{2} \leq b_{2} - \a c_{3} +\b c_{2},\\
    \dots \\
    a_{N} \leq b_N  - \a c_{N+1} + \b c_N.
  \end{gather*}
Adding all these inequalities yields
\begin{equation*}
  \sum_{n=1}^N a_n \leq \sum_{n=1}^N b_n -(\a-\b)\sum_{n=2}^N c_n -\a c_{N+1} + \b c_1  \leq \sum_{n=1}^\infty b_n  + \b c_1.
\end{equation*}
Since $N$ is arbitrary, we see that series $\sum_{n=1}^\infty a_n$ is convergent and hence $a_n \to 0$. \qed
\end{proof}

\section{Algorithm and its convergence}
\label{algorithm}
Now we formally state our algorithm.
\begin{algorithm}[Hybrid algorithm without extrapolation step]\label{alg}

    \noindent 
  \begin{enumerate}
  \item Choose $x_0, z_0 \in C$ and two parameters $k>0$ and $\la > 0$.
  \item\label{step} Given the current iterate $x_n$ and $z_n$, compute
\begin{equation}
\label{z_n} z_{n+1}  = P_C(x_n-\la Az_n).
\end{equation}
If $ z_{n+1}= x_n = z_n$ then stop. Otherwise, construct sets $C_n$ and $Q_n$ as
\begin{equation}\label{C_n}
\begin{split}
  C_0 & = H,  \\
C_n & = \bigl\{ w\in H\colon \n{z_{n+1}-w}^2  \leq \n{x_n-w}^2+k\n{x_n-x_{n-1}}^2  \\ & - (1-\frac 1 k -\la L )\n{z_{n+1}-z_n}^2 + \la L \n{z_n-z_{n-1}}^2\bigr\}, \quad n\geq 1, \\
Q_0 &= H, \\
Q_n &  = \{w\in H \colon (x_n-w, x_0-x_n)\geq 0\}, \quad n\geq 1,
\end{split}
\end{equation}
and calculate
\begin{equation}
x_{n+1} = P_{C_n\cap Q_n}x_0.
\end{equation}
\item Set $n \leftarrow n+1$ and return to step \ref{step}.
  \end{enumerate}
\end{algorithm}
We remark that sets $C_n$ look a slightly complicated in contrast to \eqref{censor_strong}. However it is only for superficial examination, for a computation it does not matter. In \eqref{C_n} and  in \eqref{nadezhda} both $C_n$ are some halfspaces.

First we show that the stopping criterion in Algorithm \ref{alg} is valid.
\begin{lemma}
  If $z_{n+1} = x_n = z_n $ in Algorithm \ref{alg} then $x_n \in S$.
\end{lemma}
\begin{proof}
From \eqref{z_n} and Lemma \ref{proj} follows
$$\la( Ax_n, y - x_n)\geq 0\quad \forall y\in C.$$
Since $\la >0$, we have $x_n\in S$.
\qed
\end{proof}

Next lemma is central to our proof of the convergence theorem.
\begin{lemma}\label{main_lemma}
  Let $(x_n)$ and $(z_n)$ be two sequences generated by Algorithm \ref{alg} and let $z\in S$. Then
\begin{multline}
\n{z_{n+1}-z}^2  \leq \n{x_n-z}^2+k\n{x_n-x_{n-1}}^2  \\ -(1-\frac 1 k -\la L )\n{z_{n+1}-z_n}^2 + \la L \n{z_n-z_{n-1}}^2 .
\end{multline}
\end{lemma}

\begin{proof}
By Lemma \ref{proj}  we have
\begin{multline}\label{fst}
  \n{z_{n+1}-z}^2 \leq \n{x_n-\la Az_n -z}^2 - \n{x_n-\la Az_n-z_{n+1}}^2  \\ = \n{x_n-z}^2 - \n{x_n-z_{n+1}}^2 -
 2\la (Az_n,z_{n+1}-z).
  \end{multline}
Since $A$ is monotone and $z\in S$, we see that $$(Az_n, z_n- z)\geq 0.$$ Thus, adding $2\la(Az_n, z_n- z)$ to the right side of \eqref{fst} we get
\begin{multline}\label{snd}
  \n{z_{n+1}-z}^2 \leq  \n{x_n-z}^2 - \n{x_n-z_{n+1}}^2 -  2\la (Az_n,z_{n+1}-z_n)  \\ = \n{x_n-z}^2 - \n{x_n-x_{n-1}}^2 - 2(x_n-x_{n-1}, x_{n-1}-z_{n+1}) \\ - \n{x_{n-1}-z_{n+1}}^2   -  2\la (Az_n,z_{n+1}-z_n) = \n{x_n-z}^2 \\ - \n{x_n-x_{n-1}}^2 - 2(x_n-x_{n-1}, x_{n-1}-z_{n+1})   -\n{x_{n-1}-z_n}^2 \\ - \n{z_n-z_{n+1}}^2 - 2(x_{n-1}-z_n,z_n-z_{n+1})  - 2\la (Az_n-Az_{n-1},z_{n+1}-z_n) \\ - 2\la (Az_{n-1},z_{n+1}-z_n) = \n{x_n-z}^2 - \n{x_n-x_{n-1}}^2  \\ - 2(x_n-x_{n-1}, x_{n-1}-z_{n+1}) - \n{x_{n-1}-z_n}^2 - \n{z_n-z_{n+1}}^2 \\  - 2\la(Az_n-Az_{n-1}, z_{n+1}-z_n) + 2(x_{n-1}-\la Az_{n-1}- z_n,z_{n+1}-z_{n}).
\end{multline}
As $z_{n} =  P_C(x_{n-1}-\la Az_{n-1})$ and $z_{n+1}\in C$ we have
\begin{equation}\label{eq:1}
(x_{n-1}-\la Az_{n-1}- z_n,z_{n+1}-z_{n}) \leq 0.
\end{equation}
Using the triangle, the Cauchy-Schwarz,  and the Cauchy inequalities we obtain
\begin{multline}\label{eq:2}
2(x_n-x_{n-1}, x_{n-1}-z_{n+1})\leq 2\n{x_n-x_{n-1}}\n{x_{n-1}-z_{n}}  +2\n{x_{n}-x_{n-1}}  \n{z_{n}-z_{n+1}}  \\ \leq
\n{x_n-x_{n-1}}^2 + \n{x_{n-1}-z_{n}}^2  + k\n{x_n-x_{n-1}}^2 + \frac 1 k \n{z_{n+1}-z_n}^2.
\end{multline}
Since $A$ is Lipschitz-continuous, we get
\begin{multline}\label{eq:3}
2\la (Az_n-Az_{n-1}, z_{n+1}-z_n) \leq 2 \la L\n{z_n-z_{n-1}} \n{z_{n+1}-z_{n}}  \\ \leq  \la L(\n{z_{n+1}-z_n}^2 + \n{z_n-z_{n-1}}^2).
\end{multline}
Combining  inequalities \eqref{snd}~--~\eqref{eq:3}, we see that
\begin{multline*}
  \n{z_{n+1}-z}^2 \leq  \n{x_n-z}^2 + k\n{x_n-x_{n-1}}^2  -  (1-\frac 1 k -\la L)\n{z_{n+1}-z_n}^2 +  \la L \n{z_{n}-z_{n-1}}^2,
\end{multline*}
which completes the proof.
\qed
\end{proof}

Now we can state and prove our main convergence result.
\begin{theorem}\label{th_1}
Assume that (C1)--(C3) hold and let  $\la \in (0,\frac 1 {2L})$, $k>\frac{1}{1-2\la L}$. Then the sequences $(x_n)$ and $(z_n)$ generated by Algorithm \ref{alg} converge strongly to $P_S x_0$.
\end{theorem}

\begin{proof}
It is evident that sets $C_n$ and $Q_n$ are closed and convex. By Lemma \ref{main_lemma} we have that $S\subseteq C_n$ for all $n\in \Z^+$. Let us show by mathematical induction that $S\subseteq Q_n$ for all $n\in \Z^+$. For $n=0$ we have $Q_0 = H$. Suppose that $S \subseteq Q_n$. It is sufficient to show that $S \subseteq Q_{n+1}$. Since  $x_{n+1}=P_{C_n\cap Q_n}x_0$ and $S\subseteq C_n\cap Q_n$, it follows that $(x_{n+1}-z, x_0-x_{n+1})\geq 0$ $\forall z\in S$. From this by the definition of $Q_n$ we conclude that $z\in Q_{n+1}$ $\forall z \in S$. Thus, $S\subseteq Q_{n+1}$ and hence $S\subseteq C_n\cap Q_n $ for all $ n\in \Z^+$. For this reason the sequence $(x_n)$ is defined correctly.

Let $\x = P_S x_0$. Since $x_{n+1}\in C_n\cap Q_n$ and $\x \in S\subseteq C_n\cap Q_n$, we have
\begin{equation} \label{bounded}
 \n{x_{n+1} - x_0}\leq \n{\x-x_0}.
\end{equation}
Therefore, $(x_n)$ is bounded. From $x_{n+1}\in C_n\cap Q_n\subseteq Q_n$ and $x_n = P_{Q_n}x_0$, we obtain
\begin{equation*}
\n{x_n-x_0}\leq \n{x_{n+1}-x_0}.
\end{equation*}
Hence, there exists $\lim_{n\r \infty}\n{x_n-x_0}$. In addition, since $x_n = P_{Q_n}x_0$ and $x_{n+1}\in Q_n$ by Lemma \ref{proj} we have
\begin{equation}\label{key}
\n{x_{n+1}-x_n}^2 \leq \n{x_{n+1}-x_0}^2 - \n{x_n-x_0}^2.
\end{equation}
From this it may be concluded that series $\sum_{n=1}^\infty \n{x_{n+1}-x_n}^2$ is convergent. In fact, from \eqref{key} and \eqref{bounded} we obtain
\begin{multline*}
\sum_{n=1}^N \n{x_{n+1}-x_n}^2\leq \sum_{n=1}^N (\n{x_{n+1}-x_0}^2 - \n{x_n-x_0}^2) \leq \n{\x-x_0}^2 - \n{x_1-x_0}^2.
\end{multline*}

Since $x_{n+1}\in C_n$, we see that
\begin{multline*}
  \n{z_{n+1}-x_{n+1}}^2 \leq  \n{x_{n+1}-x_{n}}^2 + k\n{x_n-x_{n-1}}^2 \\ -  (1-\frac 1 k -\la L)\n{z_{n+1}-z_n}^2 +  \la L \n{z_{n}-z_{n-1}}^2.
\end{multline*}
Set
\begin{gather*}
 a_n = \n{z_{n+1}-x_{n+1}}^2, \quad
 b_n = \n{x_{n+1}-x_n}^2 + k\n{x_n-x_{n-1}}^2, \\
 c_n =\n{z_{n}-z_{n-1}}^2, \quad
\a = (1-\frac 1 k -\la L), \quad
 \b = \la L.
\end{gather*}
By Lemma \ref{main_lemma} since $\sum_{n=1}^{\infty} b_n <+\infty $ and $\a > \b$,
\begin{equation*}
  \lim_{n\r \infty}\n{z_{n}-x_{n}} = 0.
\end{equation*}
For this reason $(z_n)$ is bounded and
$$\n{z_{n+1}-z_{n}}\leq \n{z_{n+1}-x_{n+1}} + \n{x_{n+1}-x_n} + \n{x_n-z_n} \r 0.$$
As $(x_n)$ is bounded there exist a subsequence $(x_{n_i})$ of $(x_n)$ such that $(x_{n_i})$ converges weakly to some  $x^*\in H$. We  show $x^* \in S$.

From \eqref{z_n} by Lemma \ref{proj} it follows that
$$(z_{n_i+1}-x_{n_i} + \la Az_{n_i}, y-z_{n_i+1})\geq 0\quad \forall y\in C.$$
This is equivalent to
\begin{multline}\label{weak}
0\leq (z_{n_i+1}-z_{n_i} + z_{n_i}-x_{n_i}, y-z_{n_i+1})+\la (Az_{n_i}, y-z_{n_i}) \\ + \la(Az_{n_i}, z_{n_i}-z_{n_i+1}) \leq
(z_{n_i+1}-z_{n_i},y-z_{n_i+1}) + (z_{n_i}-x_{n_i}, y-z_{n_i+1}) \\ + \la (Ay, y-z_{n_i}) + \la(Az_{n_i}, z_{n_i}-z_{n_i+1}) \quad \forall y\in C.
\end{multline}
In the last inequality we used monotonicity of $A$. 

Taking the limit as $i\r \infty $ in \eqref{weak} and using that $z_{n_i}\rr x^*$, we obtain
$$0\leq (Ay, y-x^*) \quad \forall y\in C,$$
which implies by Lemma \ref{minty} that $x^* \in S$.

Let us show $x_{n_i}\r x^*$.  From $\x = P_S x_0$ and $x^*\in S$ it follows that
$$\n{\x - x_0}\leq \n{x^* - x_0}\leq \liminf_{i\r \infty}\n{x_{n_i}-x_0} = \lim_{i\r \infty}\n{x_{n_i}-x_0}\leq \n{\x - x_0}.$$
Thus, $$\lim_{i\r \infty}\n{x_{n_i}-x_0} = \n{x^* - x_0}.$$
From this and $x_{n_i}-x_0 \rr x^*-x_0$ by Lemma \ref{kadec} we can conclude that $x_{n_i}-x_0 \r x^*-x_0$. Therefore, $x_{n_i}\r x^*$.

Next, we have
$$\n{x_{n_i}-\x}^2 = (x_{n_i}-x_0, x_{n_i}-\x) + (x_0-\x, x_{n_i}-\x)\leq (x_0-\x, x_{n_i}-\x). $$
As $i\r \infty$, we obtain
$$\n{x^*-\x}^2 \leq (x_0-\x, x^*-\x)\leq 0.$$
Hence we have $x^* = \x$. Since the subsequence $(x_{n_i})$ was arbitrary, we see that $x_n\r \x$. It is clear that $z_n \r \x$. \qed
\end{proof}

\section{Computational Experience}
In this section we compare the performance of  Algorithms 1, 3 and 4 which provide a strong convergence in Hilbert space. We have already noted that the main advantage of our algorithm  is a computation on every step only one value of operator. But we will show that it has competitive performance even on simple examples where the previous reasoning has no impact.  Of course all examples are considered in $\R^m$, therefore there is no sense to use any of algorithms 1,2,3,4 to obtain the solution of variational inequality. However there are many problems that arise in infinite dimensional spaces. In such problems norm convergence is often much more desirable than weak convergence (see \cite{baucomb_weak} and references therein). For this reason strong algorithms 1,2,3,4, can be better than  extragradient algorithms that provide weak convergence. Another reason to study their convergence is an academic interest.  

Although Algorithm 2 gave birth to all strong algorithms mentioned above,  we do not report its results since Algorithm 2 requires a large amount of projections onto a feasible set which are not trivial for our test problems. 

 We describe the test details below. Computations were performed using Wolfram Mathematica 8.0 on an Intel Core i3-2348M \@2.3GHz running 64-bit Linux Mint 13. 

In all tests we take $\lambda = 1/(4L)$,  $k= 3$ in Algorithm 1 and $\a_n = 0.5$  in Algorithm 3. The projection onto a feasible set $C$ is performed using the cyclic projection method with error $\varepsilon$.  The time is measured in seconds using the intrinsic Mathematica function Timing. The termination criterion is $\n{x-P_C(x-\la Ax)} \leq \varepsilon$. 
The projection in Algorithms 1, 3, 4 onto the intersection $C_n\cap Q_n$ of two halfspaces is explicit and is given by Haugazeau formula (see \cite{baucomb} for more details).

The first example is very simple. The feasible set $C\subset \R^5$ is a random polygon with $10$ linear constraints and $Ax = x-u$ where  $u$ is a random point in $\R^5$.  It is clear that the solution $x^*$ of variational inequality \eqref{vip} is the point $P_C u$. For this problem we take $L=1$ and $\varepsilon = 0.01$. We study this problem for a different choice of starting points $x_0$. The results for a three different choices of starting point $x_0$ are shown in Table 1. 

\begin{center}
\begin{tabular}{cc|c|c|c|c|c|l}
\cline{2-7}
&  \multicolumn{2}{|c|}{Ex. 1} &  \multicolumn{2}{c|}{Ex. 2} & \multicolumn{2}{c|}{Ex. 3} \\ 
\cline{2-7} \multicolumn{1}{c|}{} & Iter. & Time  &  Iter. & Time & Iter. & Time & \\ \cline{1-7}
\multicolumn{1}{ |c| }{Alg. 1} & 307 & 61.1&271  &52.1 &137  &29.4 \\ \cline{1-7}
\multicolumn{1}{ |c| }{Alg. 3}  &317 &60.4 & 315& 58.1& 276 &52.5 \\ \cline{1-7}
\multicolumn{1}{ |c| }{Alg. 4} & 518 &97.1 & 312& 58.9 &221 & 42.9 \\ \cline{1-7}
\end{tabular}
\\ \vspace{0.2cm}
\title{Table 1}
\end{center}
Of course this result depends on a feasible set. However our purpose was only to study some examples where Algorithm 1 is effective. 
 
The second problem is the following. We take $Ax = Mx + q$ with the matrix $M$ randomly generated as suggested in \cite{hphard}:
$$M = AA^T + B + D,$$ where every entry of the $n \times n$ matrix $A$ and of the $n\times n$ skew-symmetric matrix $B$ is uniformly generated from $(-5, 5)$, and every diagonal entry of the $n\times n$ diagonal $D$ is uniformly generated from $(0, 0.3)$ (so $M$ is positive definite), with every entry of $q$ uniformly generated from $(-500, 0)$. The feasible set is 
$$C = \{x\in \R^m_+ \mid x_1 + x_2 +\dots + x_m = m\},$$ and the starting point is $x_0 = (1,\dots, 1)$. 
For this problem we take $L = \n{M}$. We have generated three random samples with different choice of $M$ and $q$ for both cases $n=2$ and $n=10$, the results are tabulated in Table 2 and 3.

For $n=2$,   $\varepsilon = 0.001$

\begin{center}
\begin{tabular}{cc|c|c|c|c|c|l}
\cline{2-7}
&  \multicolumn{2}{|c|}{Ex. 1} &  \multicolumn{2}{c|}{Ex. 2} & \multicolumn{2}{c|}{Ex. 3} \\ 
\cline{2-7} \multicolumn{1}{c|}{} & Iter. & Time  &  Iter. & Time & Iter. & Time & \\ \cline{1-7}
\multicolumn{1}{|c| }{Alg. 1} &245& 4.9  &216 & 34.7 & 385 & 463.7 & \\ \cline{1-7}
\multicolumn{1}{|c| }{Alg. 3}  & 278 & 5.7 & 219 &35.1 & 349 &412.8 & \\ \cline{1-7}
\multicolumn{1}{|c| }{Alg. 4}  & 291 & 6.0 & 187 & 30.0 & 268& 320.0 &\\ \cline{1-7}
\end{tabular}
\\ \vspace{0.2cm}
\title{Table 2}
\end{center}

And for  $n=10$,  $\varepsilon = 0.01$
\begin{center}
\begin{tabular}{cc|c|c|c|c|c|l}
\cline{2-7}
&  \multicolumn{2}{|c|}{Ex. 1} &  \multicolumn{2}{c|}{Ex. 2} & \multicolumn{2}{c|}{Ex. 3} \\ 
\cline{2-7} \multicolumn{1}{c|}{} & Iter. & Time  &  Iter. & Time & Iter. & Time & \\ \cline{1-7}
\multicolumn{1}{ |c| }{Alg. 1} &4267 & 123.6  &2728  &32.7  & 554 & 15.3 & \\ \cline{1-7}
\multicolumn{1}{ |c| }{Alg. 3}  &2571 &75.9 & 1826& 29.2  & 409 & 12.1 & \\ \cline{1-7}
\multicolumn{1}{ |c| }{Alg. 4} & 2564 &76.4 & 1632&  48.3 & 522 & 16.0&\\ \cline{1-7}
\end{tabular}
\\ \vspace{0.2cm}
\title{Table 3}
\end{center}
As one can see, on our examples Algorithm 1 has competitive performance. We caution, however, that this study is very preliminary.

\bibliographystyle{siam}     
\bibliography{publications}

\end{document}